\documentclass[11pt,a4paper]{amsart}
\usepackage[T1]{fontenc}
\usepackage{lmodern}
\usepackage[textwidth=150mm,textheight=225mm,centering]{geometry}
\usepackage{mathtools,amssymb}
\usepackage{booktabs,array,float}
\usepackage{microtype}
\usepackage[numbers,sort&compress]{natbib}
\usepackage{xurl}
\usepackage[hidelinks]{hyperref}
\hypersetup{pdftitle={Low-Twist Matrix Covariants of Exterior Powers: Vanishing and Modular Phenomena},pdfauthor={Jingchuan Ma},pdfsubject={Research manuscript for Algebras and Representation Theory},pdfkeywords={polynomial representations, exterior powers, plethysm, Schur algebras, divided powers, modular representations}}
\newtheorem{theorem}{Theorem}[section]
\newtheorem{proposition}[theorem]{Proposition}
\newtheorem{lemma}[theorem]{Lemma}

\theoremstyle{remark}
\newtheorem{remark}[theorem]{Remark}
\newcommand{\GL}{\operatorname{GL}}
\newcommand{\End}{\operatorname{End}}
\newcommand{\Sym}{\operatorname{Sym}}
\newcommand{\Hom}{\operatorname{Hom}}
\newcommand{\Ind}{\operatorname{Ind}}
\newcommand{\Inf}{\operatorname{Inf}}
\newcommand{\Id}{\operatorname{Id}}
\newcommand{\wt}{\operatorname{wt}}
\newcommand{\sgnchar}{\operatorname{sgn}}
\newcommand{\rank}{\operatorname{rank}}
\newcommand{\cH}{\mathcal H}
\newcommand{\eps}{\varepsilon}
\newcommand{\F}{\mathbb F}
\newcommand{\Q}{\mathbb Q}
\newcommand{\Z}{\mathbb Z}
\newcommand{\doi}[1]{\url{https://doi.org/#1}}
\numberwithin{equation}{section}
\allowdisplaybreaks[1]
\title[Low-Twist Matrix Covariants of Exterior Powers]{Low-Twist Matrix Covariants of Exterior Powers:\\Vanishing and Modular Phenomena}
\author[J. Ma]{Jingchuan Ma}
\thanks{Jingchuan Ma (corresponding author), Fuzhou University Zhicheng College, Fuzhou 350002, Fujian, China. Email: \href{mailto:kiciot@qq.com}{\texttt{kiciot@qq.com}}. ORCID: \href{https://orcid.org/0009-0001-3703-471X}{0009-0001-3703-471X}}
\date{}

\begin{document}
\begin{abstract}
We study morphisms from symmetric powers of exterior powers to determinant-twisted endomorphism representations of general linear groups. At the minimal positive determinant twist, we prove vanishing over every field of characteristic different from two: if the exterior degree $r\geq3$ is odd, the symmetric degree satisfies $d\geq3$, and the underlying space has dimension $rd$, then the corresponding equivariant Hom space is zero. The proof uses a block-exchange sign and a universal root-subgroup identity, so it also applies in small odd characteristics without semisimplicity. We then determine the second-twist spaces for trivectors in characteristic zero. In dimension $3m$ and degree $2m$, they are scalar and one-dimensional for $m=2$, and zero for $m\geq3$. Plethystic conjugation reduces the latter vanishing to an elementary weight-support bound for exterior powers of the ten-dimensional space of ternary cubics. Finally, over fields of characteristic zero or odd characteristic, an exact computer-assisted classification in dimension nine gives a one-dimensional scalar Hom space in characteristic five and zero in characteristic zero and in every odd characteristic other than five.
\end{abstract}
\maketitle

\section{Introduction}\label{sec:introduction}
Let $V$ be an $n$-dimensional vector space over a field $k$. We consider
\begin{equation}\label{eq:H}
 \cH_{n,r,d,\ell}(V)
 =\Hom_{\GL(V)}\!\left(\Sym^d(\bigwedge^r V),
                  \End(V)\otimes\det(V)^\ell\right),
\end{equation}
where $\GL(V)$ acts on $\End(V)$ by conjugation. Here $\Sym^d$ is the ordinary coinvariant symmetric power, and the Hom is taken between polynomial representations of the algebraic group scheme $\GL(V)$. We use \emph{matrix covariant} for a morphism in \eqref{eq:H}. In small positive characteristic this is not the same object as an ordinary coordinate-polynomial map on $\bigwedge^r V$: diagonal evaluation can annihilate a nonzero morphism.

The action of the scalar torus gives the necessary condition
\begin{equation}\label{eq:balance-intro}
 rd=n\ell.
\end{equation}
This condition organizes the two positive determinant twists considered here. Our main result concerns the minimal twist and does not rely on a characteristic-zero character calculation.

\begin{theorem}[Vanishing at the minimal twist]\label{thm:minimal}
Let $k$ be a field with $\operatorname{char}k\neq2$. Suppose that $r\geq3$ is odd, $d\geq3$, and $\dim V=rd$. Then
\[
 \cH_{rd,r,d,1}(V)=0.
\]
\end{theorem}

In characteristic zero, the relevant hook-plethysm vanishing is covered by Langley--Remmel \cite{LangleyRemmel2004}. The proof given here applies to actual Hom spaces in every odd characteristic, including those in which the polynomial representations are not semisimple. On the determinant-weight space, an exchange of two odd exterior blocks fixes a source monomial but changes the sign of an appropriate target coefficient. A universal root-subgroup identity then forces all remaining target coefficients to vanish. The argument is over the base field and over $k[t]$, not over the finite set of field-valued points. The classical degree-two moment map and the even-degree wedge map locate the boundaries of the statement.

The second twist for trivectors has $d=2m$ and $n=3m$. In characteristic zero its two possible target constituents are indexed by
\[
 \alpha_m=(2^{3m}),\qquad \beta_m=(3,2^{3m-2},1).
\]
Proposition~\ref{prop:second-twist} determines their multiplicities in $h_{2m}[e_3]$: the scalar coefficient is one at $m=2$ and zero for $m\geq3$, while the other coefficient is always zero. The boundary $m=2$ is known, including the non-scalar vanishing in de Boeck's thesis \cite[p.~55]{deBoeck2015Thesis}. For $m\geq3$, plethystic conjugation gives a short weight-support proof: in an exterior product of $2m$ distinct ternary cubics, at least $2m-4$ factors must involve any prescribed variable. The required highest weights have third coordinate zero or one. We present this characteristic-zero classification as a companion to Theorem~\ref{thm:minimal}.

An isolated modular companion occurs at $(n,r,d,\ell)=(9,3,6,2)$. Theorem~\ref{thm:modular} shows that \eqref{eq:H} is zero in characteristic zero and every odd characteristic other than five. In characteristic five it is one-dimensional and scalar-valued. This result is supported by an integral root-equation matrix and exact certificates. The exceptional scalar component lies in
\[
 \Gamma^6\!\left((\bigwedge^3V)^*\right)\otimes\det(V)^2,
\]
and its value on every pure sixth power is zero. It therefore gives no nonzero ordinary degree-six coordinate invariant.

Matrix covariants are also natural candidates for producing conjugacy-compatible objects in form-equivalence problems. This motivation does not enlarge the scope of the results: only the Hom spaces explicitly stated here are classified, and no algorithmic lower bound or exclusion of other invariant methods is claimed.

Section~\ref{sec:preliminaries} fixes the representation-theoretic conventions. Section~\ref{sec:minimal} proves Theorem~\ref{thm:minimal} and treats its boundaries. Section~\ref{sec:second} gives the characteristic-zero second-twist calculation. Section~\ref{sec:modular} treats the characteristic-five class, and Appendix~\ref{app:certificate} gives the finite reduction and computational certificate.

\section{Polynomial representations and weights}\label{sec:preliminaries}
Throughout, $n,r,d\geq1$ and $\ell\geq1$. The source in \eqref{eq:H} is the ordinary symmetric power, not the divided power. The target action is
\begin{equation}\label{eq:target-action}
 g\cdot(X\otimes\eta)=(gXg^{-1})\otimes\det(g)^\ell\eta
\end{equation}
after expressing the determinant line in a fixed basis. The target is polynomial for $\ell\geq1$, since $V^*\otimes\det(V)\cong\bigwedge^{n-1}V$.

For a partition $\lambda$, write $S_\lambda V$ for the corresponding Schur module and $s_\lambda$ for its Schur function. In characteristic zero the universal character of $\Sym^d(\bigwedge^r V)$ is $h_d[e_r]$; specialization to $n$ variables gives the character on $V$. We use the Hall inner product for which the Schur functions are orthonormal, with the conventions of Macdonald \cite{Macdonald1995}.

For $n\geq2$ and in characteristic zero,
\begin{equation}\label{eq:target-split}
 \End(V)\otimes\det(V)^\ell
 \cong S_{(\ell^n)}V\oplus S_{(\ell+1,\ell^{n-2},\ell-1)}V.
\end{equation}
The first summand is the scalar line. This decomposition will be used only in characteristic zero. In particular, when $\operatorname{char}k$ divides $n$, the scalar line is contained in the trace-zero subspace, so it is not a complementary summand to that subspace.

For arbitrary $k$, equivariance means an identity of rational representations of the algebraic group scheme. In fixed homogeneous polynomial degree this is the corresponding Schur-algebra Hom \cite{Green1980}. Strict polynomial functors provide a compatible framework \cite{FriedlanderSuslin1997,AquilinoReischuk2017}, but we work with the displayed finite-dimensional modules. Over finite fields these conditions are stronger than equivariance for the finite abstract group of field-valued points.

\begin{proposition}[Scalar weight constraint]\label{prop:balance}
If $\cH_{n,r,d,\ell}(V)\neq0$, then $rd=n\ell$.
\end{proposition}
\begin{proof}
The scalar torus acts on the source with character $z^{rd}$ and on the target with character $z^{n\ell}$. A nonzero equivariant map forces equality of these torus characters.
\end{proof}

We shall also use the parity-sensitive conjugation identity for plethysm \cite[Eq.~(1)]{LangleyRemmel2004}. Let $\omega(p_j)=(-1)^{j-1}p_j$, so that $\omega(s_\lambda)=s_{\lambda'}$.
\begin{lemma}[Plethystic conjugation]\label{lem:omega}
If $g$ is homogeneous of degree $r$, then
\begin{equation}\label{eq:omega}
 \omega(f[g])=\omega^r(f)[\omega(g)],
\end{equation}
where $\omega^r$ is $\omega$ for odd $r$ and the identity for even $r$. In particular, $\omega(h_q[e_3])=e_q[h_3]$.
\end{lemma}
\begin{proof}
It is enough to check $f=p_j$ and then multiply over parts of a power-sum monomial. If $g=\sum_{\mu\vdash r}c_\mu p_\mu$, then
\[
 \omega(p_j[g])
 =\sum_{\mu\vdash r}c_\mu(-1)^{jr-\ell(\mu)}p_{j\mu}
 =(-1)^{r(j-1)}p_j[\omega(g)].
\]
This is \eqref{eq:omega}; the last assertion uses $\omega(h_q)=e_q$ and $\omega(e_3)=h_3$.
\end{proof}

\section{Vanishing at the minimal twist}\label{sec:minimal}

We prove Theorem~\ref{thm:minimal} by working directly with the torus weights and root-subgroup actions.

Put $n=rd$, fix a basis $e_1,\ldots,e_n$ of $V$, and set
$\varpi=e_1\wedge\cdots\wedge e_n$.  For an increasing $r$-subset
$A=\{a_1<\cdots<a_r\}$, write
$e_A=e_{a_1}\wedge\cdots\wedge e_{a_r}$.  Monomials
\[
 e_{A_1}\cdots e_{A_d},\qquad A_1\leq\cdots\leq A_d,
\]
form a basis of $M=\Sym^d(\bigwedge^rV)$.  Their diagonal-torus weight is
the occurrence-count vector of the coordinates among the blocks.

Let $E_{ab}$ be the matrix unit satisfying $E_{ab}e_b=e_a$, and put
$\delta=(1,\ldots,1)$.  The target weights are
\begin{equation}\label{eq:target-weights-minimal}
 \wt(E_{ab}\otimes\varpi)=\delta+\eps_a-\eps_b.
\end{equation}
Thus only the determinant weight $\delta$ and the defect weights
$\delta+\eps_i-\eps_j$ can contribute.  This is a split-torus comodule
calculation, not an equality of characters on field-valued points.

The determinant-weight monomials are indexed by partitions of $\{1,\ldots,n\}$ into $d$ unordered $r$-subsets. For $P=\{A_1,\ldots,A_d\}$, write $m_P=\prod_{A\in P}e_A$.

\begin{lemma}[Odd block swap]\label{lem:block-swap}
Assume $r$ is odd, $d\geq3$, and $\operatorname{char}k\neq2$.  If
\[
 F:M\longrightarrow\End(V)\otimes\det(V)
\]
is $\GL(V)$-equivariant, then $F$ vanishes on the weight space $M_\delta$.
\end{lemma}

\begin{proof}
Fix $P$ and a coordinate $i$.  The unique block containing $i$ leaves at
least two blocks $A=\{a_1<\cdots<a_r\}$ and
$B=\{b_1<\cdots<b_r\}$ that avoid $i$.  Let
\[
 \sigma=(a_1\ b_1)\cdots(a_r\ b_r).
\]
The order-preserving matching gives $\sigma(e_A)=e_B$ and
$\sigma(e_B)=e_A$ without an internal wedge sign.  All other blocks are
fixed, and multiplication in $\Sym^d$ is commutative, so
$\sigma(m_P)=m_P$.  On the other hand,
$\det(\sigma)=(-1)^r=-1$.

Weight preservation gives
$F(m_P)=\sum_{a=1}^n c_aE_{aa}\otimes\varpi$.  Since $\sigma$ fixes $i$,
conjugation fixes $E_{ii}$, whereas the determinant factor changes sign.
Equivariance therefore gives $c_i=-c_i$.  The hypotheses on the
characteristic imply $c_i=0$.  Since $P$ and $i$ were arbitrary, the
assertion follows.
\end{proof}

\begin{lemma}[Universal root step]\label{lem:root-step}
Under the hypotheses of Lemma~\ref{lem:block-swap}, an equivariant $F$ also
vanishes on every defect-weight space
$M_{\delta+\eps_i-\eps_j}$, $i\neq j$.
\end{lemma}

\begin{proof}
Fix $i\neq j$ and $x\in M_{\delta+\eps_i-\eps_j}$.  By
\eqref{eq:target-weights-minimal},
\[
 F(x)=cE_{ij}\otimes\varpi
\]
for some $c\in k$.  Work over $k[t]$ with the root-subgroup element
\[
 u_{ji}(t)=I+tE_{ji}.
\]
Every monomial in $x$ contains $i$ twice and $j$ zero times.  The
coefficient of $t$ in $u_{ji}(t)x$ is therefore obtained by replacing one
occurrence of $i$ by $j$; it lies in $M_\delta$.  Cancellation can make
this coefficient zero but cannot change its weight.

The exact target identity is
\begin{equation}\label{eq:root-identity}
 u_{ji}(t)E_{ij}u_{ji}(t)^{-1}
 =E_{ij}+t(E_{jj}-E_{ii})-t^2E_{ji},
\end{equation}
and $\det u_{ji}(t)=1$.  Base-change $F$ to $k[t]$ and compare the
coefficient of $t$ in
$F(u_{ji}(t)x)=u_{ji}(t)F(x)$.  Lemma~\ref{lem:block-swap} kills the
coefficient on the left, while \eqref{eq:root-identity} gives
\[
 0=c(E_{jj}-E_{ii})\otimes\varpi.
\]
The two matrix units are linearly independent over every field, hence
$c=0$.
\end{proof}

\begin{proof}[Proof of Theorem~\ref{thm:minimal}]
Lemmas~\ref{lem:block-swap} and \ref{lem:root-step} kill every source
weight for which the target has a nonzero weight space.  All remaining
weights map to zero by torus equivariance.
\end{proof}

The coefficient comparison is an identity over $k[t]$, so it remains valid over finite base fields. The only use of $\operatorname{char}k\neq2$ is in Lemma~\ref{lem:block-swap}.

\begin{remark}[Exact boundaries]\label{rem:minimal-boundaries}
If $r$ is even, the block swap has determinant $+1$, and the conclusion is
false: graded commutativity makes wedge multiplication a nonzero map
\[
 \Sym^d(\textstyle\bigwedge^rV)\longrightarrow\det(V),
 \qquad n=rd,
\]
which becomes scalar matrix-valued after composing with
$\eta\mapsto\Id_V\otimes\eta$.  In characteristic two, the sign obstruction
in Lemma~\ref{lem:block-swap} disappears.  No characteristic-two conclusion
is asserted.
\end{remark}

In characteristic zero, the non-scalar target is a hook, and the vanishing follows from Langley--Remmel hook plethysm \cite[Theorem~3.1]{LangleyRemmel2004}. Semisimplicity also transfers this fixed-degree conclusion when the characteristic exceeds $rd$. The proof above covers the remaining small odd characteristics without a filtration hypothesis. For general modular structure of symmetric and exterior powers, see \cite{Donkin2001,HagueMcNinch2013}.

\subsection{Low-degree boundary}
The restriction $d\geq3$ has classical exceptions. If $d=1$ and $n=r$, the source is $\det(V)$ and the scalar inclusion
\[
 \eta\longmapsto\Id_V\otimes\eta
\]
is nonzero.

Let $d=2$, $n=2r$, assume $r$ is odd, and suppose that $2$ is invertible in $k$. Wedge multiplication is an alternating $\det(V)$-valued pairing on $W=\bigwedge^rV$. Together with the infinitesimal $\mathfrak{gl}(V)$-action and the nondegenerate trace pairing on $\End(V)$, it induces a linear map
\[
 \mu:\Sym^2W\longrightarrow\End(V)\otimes\det(V),
\]
characterized by the symmetric bilinear formula below, where $\mu(\alpha,\beta)=\mu(\alpha\beta)$:
\begin{equation}\label{eq:moment-map}
 \operatorname{tr}(\mu(\alpha,\beta)X)
 =\tfrac12\bigl((X\alpha)\wedge\beta+(X\beta)\wedge\alpha\bigr).
\end{equation}
This construction is equivariant. For complementary basis wedges $\alpha=e_1\wedge\cdots\wedge e_r$ and $\beta=e_{r+1}\wedge\cdots\wedge e_{2r}$, taking $X=E_{11}$ makes the right side of \eqref{eq:moment-map} equal to $\tfrac12\alpha\wedge\beta\neq0$, proving nonzeroness. This is the standard moment-map boundary; for $r=3$ it recovers Hitchin's six-dimensional endomorphism up to normalization \cite{Hitchin2000}. Related exterior-form constructions are discussed in \cite{Rubtsov2019}. These boundary constructions and the even-$r$ wedge map in Remark~\ref{rem:minimal-boundaries} are used as prior context, not as new results.

\section{The second twist in characteristic zero}\label{sec:second}
Assume $\operatorname{char}k=0$, take $r=3$, and put $n=3m$ and $d=2m$. These are the parameters permitted by Proposition~\ref{prop:balance} at twist two. Define
\begin{equation}\label{eq:coefficients}
 a_m=\langle h_{2m}[e_3],s_{(2^{3m})}\rangle,\qquad
 b_m=\langle h_{2m}[e_3],s_{(3,2^{3m-2},1)}\rangle.
\end{equation}
By \eqref{eq:target-split}, these are the multiplicities in $\Sym^{2m}(\bigwedge^3V)$ of the scalar and non-scalar irreducible constituents of the target of \eqref{eq:H}, respectively.

The conjugate partitions are $(3m,3m)$ and $(3m,3m-1,1)$. Lemma~\ref{lem:omega} therefore gives
\begin{equation}\label{eq:conjugated-coefficients}
 a_m=\langle e_{2m}[h_3],s_{(3m,3m,0)}\rangle,\qquad
 b_m=\langle e_{2m}[h_3],s_{(3m,3m-1,1)}\rangle.
\end{equation}
Since these partitions have at most three parts, the coefficients may be read in $\bigwedge^{2m}(\Sym^3 E)$ for a three-dimensional space $E$. Specialization to three variables does not change the multiplicity of a Schur function indexed by a partition with at most three parts.

\begin{lemma}[Weight support for exterior powers of ternary cubics]\label{lem:weight-support}
Let $E$ have basis $x,y,z$, and let $0\leq q\leq10$. Every weight $\nu$ of $\bigwedge^q(\Sym^3 E)$ satisfies $\nu_i\geq q-4$ for each coordinate $i$.
\end{lemma}
\begin{proof}
A basis of $\Sym^3E$ consists of its ten monomials $x^ay^bz^c$ with $a+b+c=3$. Exactly four have $c=0$:
\[
 x^3,\quad x^2y,\quad xy^2,\quad y^3.
\]
A basis wedge uses $q$ distinct monomials. At least $q-4$ of them involve $z$, so the total $z$ exponent is at least $q-4$. The same argument applies to the other coordinates. For $q<4$ the stated bound is automatic.
\end{proof}

\begin{proposition}[Second-twist trivector classification]\label{prop:second-twist}
For every $m\geq2$,
\[
 b_m=0,\qquad
 a_m=\begin{cases}1,&m=2,\\0,&m\geq3.\end{cases}
\]
Consequently, for $\operatorname{char}k=0$ and $\dim V=3m$, the space $\cH_{3m,3,2m,2}(V)$ is one-dimensional and scalar-valued at $m=2$, and is zero for $m\geq3$.
\end{proposition}
\begin{proof}
First suppose $3\leq m\leq5$. Lemma~\ref{lem:weight-support} gives a lower bound $2m-4\geq2$ for the third coordinate of every weight in $\bigwedge^{2m}(\Sym^3E)$. Neither $(3m,3m,0)$ nor $(3m,3m-1,1)$ is a weight. Hence neither can be a highest weight of a constituent, and \eqref{eq:conjugated-coefficients} gives $a_m=b_m=0$. For $m\geq6$ the entire exterior power is zero, since $\dim\Sym^3E=10$.

For the known boundary $m=2$, the three-variable character calculation is
\begin{equation}\label{eq:exterior4}
 \bigwedge^4(\Sym^3E)
 \cong S_{(8,3,1)}E\oplus S_{(7,4,1)}E\oplus S_{(6,6,0)}E
       \oplus S_{(6,4,2)}E\oplus S_{(6,3,3)}E.
\end{equation}
To check this identity directly, extract the coefficient of $u^4$ in
\[
 \prod_{a+b+c=3}(1+u x^ay^bz^c)
\]
and subtract Schur characters in the highest-weight order
\[
 (8,3,1),\ (7,4,1),\ (6,6,0),\ (6,4,2),\ (6,3,3).
\]
Each multiplicity is one and the residual character is zero. Their dimensions are respectively $81,64,28,27,10$, with sum $210=\binom{10}{4}$. Thus $(6,6,0)$ occurs once and $(6,5,1)$ does not occur, yielding $a_2=1$ and $b_2=0$. The latter boundary also appears in \cite[p.~55]{deBoeck2015Thesis}. The conclusion about Hom follows from \eqref{eq:target-split}.
\end{proof}

The case $m=1$ is the scalar inclusion $\det(V)^2\to\End(V)\otimes\det(V)^2$, since $\dim V=3$ and $\bigwedge^3V=\det(V)$. No positive-characteristic conclusion is inferred from the semisimple calculation above.

\subsection{Signed Foulkes interpretation}
For completeness, the sign conventions in \eqref{eq:conjugated-coefficients} can also be read at the symmetric-group level. Put $H=S_3\wr S_{2m}\leq S_{6m}$ and let
\[
 \eps_{\mathrm{base}}((\sigma_1,\ldots,\sigma_{2m});\pi)
 =\prod_{a=1}^{2m}\sgnchar(\sigma_a).
\]
The characteristic of $M_m=\Ind_H^{S_{6m}}\eps_{\mathrm{base}}$ is $h_{2m}[e_3]$. Since a block transposition exchanges three pairs of letters,
\[
 \sgnchar_{S_{6m}}|_H
 =\eps_{\mathrm{base}}\,\Inf_{S_{2m}}^H(\sgnchar_{S_{2m}}),
\]
and hence
\[
 M_m\otimes\sgnchar_{S_{6m}}
 \cong\Ind_H^{S_{6m}}\Inf_{S_{2m}}^H(\sgnchar_{S_{2m}}).
\]
Its characteristic is $e_{2m}[h_3]$, and tensoring a Specht module by sign conjugates its partition. This is the standard twisted Foulkes formulation; see \cite{deBoeck2015Thesis,PagetWildon2016}. The weight argument above is enough to settle the specified family, without a general classification of twisted Foulkes constituents.

\section{A modular scalar class}\label{sec:modular}
The actual Hom space in positive characteristic need not agree with the characteristic-zero multiplicity. The first fixed case considered here illustrates the distinction.

\begin{theorem}[Computer-assisted classification in dimension nine]\label{thm:modular}
Let $V$ have dimension nine over a field $k$ of characteristic zero or odd characteristic, and set $W=\bigwedge^3V$. Then
\begin{equation}\label{eq:modular-dimension}
 \dim\Hom_{\GL(V)}\!\left(\Sym^6W,\End(V)\otimes\det(V)^2\right)
 =\begin{cases}
 1,&\operatorname{char}k=5,\\
 0,&\operatorname{char}k=0\text{ or }\operatorname{char}k=p\neq2,5.
 \end{cases}
\end{equation}
In characteristic five every map in this space is scalar-valued.
\end{theorem}
\begin{proof}
The equivariance conditions reduce to a $475\times47$ integer matrix, as described in Appendix~\ref{app:certificate}. Eight full minors have determinant gcd $2^{14}\cdot5$, so the matrix has full column rank in characteristic zero and in every odd characteristic other than five. Its rank in characteristic five is $46$. The one-dimensional kernel gives zero off-diagonal coefficients and equal diagonal coefficients on every source basis element. These exact calculations, together with the reduction to the displayed Hom space, prove the statement. Online Resource~1 supplies the matrix, parameter map, minors, kernel and scalarity records, and replay code.
\end{proof}

The scalar component is intrinsically a morphism
\[
 F_{\mathrm{sc}}:\Sym^6W\longrightarrow\det(V)^2,
 \qquad F=\Id_V\otimes F_{\mathrm{sc}}.
\]
For finite-dimensional $W$ there is a natural identification
\begin{equation}\label{eq:divided-powers-duality}
 (\Sym^6W)^*\otimes\det(V)^2
 \cong\Gamma^6(W^*)\otimes\det(V)^2.
\end{equation}
Thus the invariant class representing $F_{\mathrm{sc}}$ lies in the space on the right, not in the ordinary coordinate-polynomial space $\Sym^6(W^*)\otimes\det(V)^2$. The distinction between divided powers and ordinary invariant polynomials is also central to \cite{Tange2025}.

Fix an ordered basis $e_0,\ldots,e_8$ and a volume element $\varpi=e_0\wedge\cdots\wedge e_8$. Define the coordinate functional $\lambda_{5,\varpi}$ by
\begin{equation}\label{eq:scalar-trivialization}
 F_{\mathrm{sc}}(v)=\lambda_{5,\varpi}(v)\varpi^2,
 \qquad
 F(v)=\lambda_{5,\varpi}(v)\Id_V\otimes\varpi^2.
\end{equation}
The normalization in Online Resource~1 gives
\begin{equation}\label{eq:witness}
 \lambda_{5,\varpi}(e_{012}^2e_{345}^2e_{678}^2)=1,
 \qquad e_{abc}=e_a\wedge e_b\wedge e_c.
\end{equation}
This proves that the abstract functional is nonzero, but it does not assert nonzero evaluation on a pure power.

\begin{proposition}[Vanishing on pure powers]\label{prop:pure-power}
In characteristic five, $\lambda_{5,\varpi}(T^6)=0$ for every $T\in W$.
\end{proposition}
\begin{proof}
Torus equivariance restricts the support to incidence weight $(2^9)$. Each trivector basis coordinate therefore occurs with multiplicity at most two. The possible multiplicity profiles of a supported degree-six monomial are
\[
 (2,2,2),\quad(2,2,1,1),\quad(2,1,1,1,1),\quad(1,1,1,1,1,1).
\]
Their multinomial coefficients in $T^6$ are $90,180,360,720$, respectively. All vanish modulo five, so every supported term disappears on the diagonal.
\end{proof}

If $\varpi'=c\varpi$ with $c\neq0$, then $\lambda_{5,\varpi'}=c^{-2}\lambda_{5,\varpi}$. Pure-power vanishing is consequently intrinsic. The nonzero class in \eqref{eq:divided-powers-duality} is not a nonzero ordinary degree-six invariant of trivectors. In characteristic zero, for $\dim V=9$, the ordinary invariant algebra $k[\bigwedge^3V]^{\operatorname{SL}(V)}$ is polynomial with generator degrees $12,18,24,30$ \cite[Proposition~2.7]{RainsSam2018}; the corresponding geometry is studied in \cite{GrusonSam2015}. These coordinate invariants and the class in Theorem~\ref{thm:modular} are different objects.

More generally, diagonal evaluation of any morphism in \eqref{eq:H} satisfies
\[
 F((gT)^d)=\det(g)^\ell\,gF(T^d)g^{-1}
\]
after trivializing the determinant line. In small characteristic this operation is not faithful, as Proposition~\ref{prop:pure-power} demonstrates. The results here concern the specified symmetric-power Hom spaces; they do not exclude ordinary polynomial covariants with a different linearization, rational covariants, or operators on other representations.

\appendix
\section{Exact certificate for Theorem~\ref{thm:modular}}\label{app:certificate}
This appendix gives the finite reduction behind the computer-assisted result. The data and code are supplied with the article as Online Resource~1. All calculations use integers or finite-field residues.

\subsection{Weights and signed permutation relations}
Use the basis and volume element fixed in Section~\ref{sec:modular}, and let $M=\Sym^6(\bigwedge^3V)$. The target weights are $2\delta$ and $2\delta+\eps_i-\eps_j$, where $\delta=(1,\ldots,1)$. Only the scalar-weight space and a fixed defect-weight space need separate parametrizations; permutations identify the other defect spaces. Multiset enumeration gives
\begin{equation}\label{eq:basis-counts}
 \dim M_{(2^9)}=128800,\qquad
 \dim M_{(3,1,2^7)}=74445.
\end{equation}
A coefficient state records a source monomial together with the target matrix entry. Signed adjacent-transposition traversal on these states gives the counts in Table~\ref{tab:orbits}. A negative-orientation stabilizer forces the corresponding coefficient to be its own negative and hence to vanish in odd characteristic.

\begin{table}[H]
\centering
\caption{Signed coefficient-state orbit reduction in odd characteristic.}\label{tab:orbits}
\begin{tabular}{lrrr}
\toprule
Coefficient type & Raw orbits & Killed by sign & Parameters\\
\midrule
Diagonal & 21 & 4 & 17\\
Fixed defect & 44 & 14 & 30\\
\midrule
Total & 65 & 18 & 47\\
\bottomrule
\end{tabular}
\end{table}
The resulting $17+30=47$ parameters give the most general torus- and permutation-equivariant map in odd characteristic. No scalar/trace-zero splitting is used.

\subsection{Universal root equations}
For $u_{01}(t)=I+tE_{01}$, impose
\begin{equation}\label{eq:universal}
 F(u_{01}(t)v)=u_{01}(t)F(v)
\end{equation}
coefficientwise through degree six on every source weight whose root expansion can meet a target weight. The source weights mapping to zero must also satisfy these equations. Permutations commuting with the root reduce the $179$ relevant profiles to $14$ canonical classes. Signed permutation equivariance then supplies every input monomial in the profiles, and permutation conjugacy supplies every root subgroup.

These are universal polynomial identities, not tests at finitely many parameter values. Together with the torus and permutation relations they give the generators of the degree-$18$ Schur-algebra action \cite{Green1980}. Exact sign deduplication yields
\begin{equation}\label{eq:matrix}
 A\in\operatorname{Mat}_{475\times47}(\Z),
 \qquad \#\{(i,j):A_{ij}\neq0\}=1250.
\end{equation}
The row encoding and its SHA-256 digest are retained in Online Resource~1. Since reduction is over the prime field and matrix rank is unchanged by scalar extension, the kernel dimensions apply over every field of the specified characteristic.

\subsection{Rank and scalarity certificates}
Eight retained $47\times47$ row minors have determinant gcd
\begin{equation}\label{eq:gcd}
 81920=2^{14}\cdot5.
\end{equation}
Any prime at which the rank drops must divide all full minors, hence also the gcd of these eight. A nonzero retained minor proves rank $47$ over $\Q$; \eqref{eq:gcd} proves full column rank in every odd characteristic other than five. Exact elimination gives
\[
 \rank_{\Q}A=47,\qquad\rank_{\F_3}A=47,\qquad\rank_{\F_5}A=46.
\]
The normalized generator of $\ker(A\bmod5)$ is
\[
 (1,2,2,2,1,1,4,4,4,4,4,2,2,2,2,1,1,
    \underbrace{0,\ldots,0}_{30}).
\]
The final thirty coordinates are precisely the off-diagonal parameters. On every one of the $128800$ scalar-weight monomials, substitution gives nine equal diagonal coefficients. Together with the zero images of all other source weights, this proves scalarity on the entire module. Equation~\eqref{eq:witness} supplies a nonzero abstract value; Proposition~\ref{prop:pure-power} explains why it is not a nonzero diagonal evaluation.

The retained certificate includes a reconstruction by a separately structured implementation. An explicit signed permutation of the $47$ parameters identifies the complete $475$-row matrices. A separate exact computer-algebra calculation reproduces the ranks and all eight minor determinants. Online Resource~1 contains the construction primitives, compressed bases, parameter map, root equations, minors, kernel and scalarity records, abstract witness, tests, and exact replay commands. The computational proof depends on the completeness of the weight/root reduction and on these exact finite calculations; it is not a proof-assistant formalization.

\section*{Supplementary information}
\noindent\textbf{Online Resource 1.} Exact integer and finite-field certificates and replay code for Theorem~\ref{thm:modular}, together with an exact check of the characteristic-zero calculation in Proposition~\ref{prop:second-twist}. The resource is supplied as \texttt{ESM\_1.zip}.

\section*{Statements and Declarations}
\subsection*{Funding}
No specific funding was received for this work.
\subsection*{Competing interests}
The author declares no competing financial or non-financial interests relevant to this work.
\subsection*{Author contribution}
Jingchuan Ma is the sole author and takes responsibility for the manuscript and its accompanying computational material.
\subsection*{Data and code availability}
The exact verification code and machine-readable certificates are supplied with this submission as Online Resource~1. No external empirical datasets were used.
\subsection*{Use of generative AI}
ChatGPT assisted in writing code for verification purposes and contributed to the translation and revision of the manuscript. The author bears full responsibility for the final version, mathematical assertions, citations, and supporting evidence.

\bibliographystyle{splncs04}
\bibliography{references}
\end{document}